\documentclass[preprint,12pt]{elsarticle}

\usepackage{hyperref}
\usepackage{amssymb}
\usepackage{mathptmx}      
\usepackage{amssymb,amsmath}
\usepackage{booktabs}
\usepackage{bm}
\usepackage [english]{babel}
\usepackage [autostyle, english = american]{csquotes}
\MakeOuterQuote{"}
\usepackage{amsthm}
\newtheorem{theorem}{Theorem}[section]

\makeatletter
\def\ps@pprintTitle{%
   \let\@oddhead\@empty
   \let\@evenhead\@empty
   \let\@oddfoot\@empty
   \let\@evenfoot\@oddfoot
}
\makeatother

\begin{document}

\begin{frontmatter}

\title{Web spline error estimation of non-cooperative elliptic equations for population dynamics}

 \author[label1]{Ayan Chakraborty}
 \address[label1]{FB:513~,~IIT Kanpur}
 \ead{achkbrty@gmail.com}
 \author[label2]{B.V.Rathish Kumar}
 \address[label2]{FB:571~,~IIT Kanpur~,~00915122597660}
 \ead[url]{$http://iitk.ac.in/new/rathish-kumar-b-v$}
  \ead{bvrathishkumar@gmail.com}


\address{Department of Mathematics and Scientific Computing \\ Indian Institute of Technology Kanpur}

\begin{abstract}
We analyze the error of the WEB-S finite element method applied to  elliptic systems with non-cooperative  dominant coupling,with a mixed Dirichlet/Neumann/Robin boundary condition. This problem is strongly related to a posteriori error estimates, giving computable bounds for computational errors and detecting zones in the solution domain where such errors are too large and certain mesh refinements should be performed. These results are based on an extensive regularity analysis of the interface problems of concern.Finally, the error analysis is illustrated by numerical experiments.

\end{abstract}


\end{frontmatter}


\section{Introduction}
The advection-diffusion equation governs several important phenomena in physics and engineering, and it is the basis of advanced fluid problems such as $Navier~Stokes$ problems. Therefore it has been the focus of intense research for quite some time.If the coefficient matrix (cf sec (\ref{sec3}))is $rough$ or $highly ~ varying$ then such problems are challenging to solve using standard finite element methods since very fine meshes are needed in order to resolve the features of the solution. Various approaches have been proposed in this field \citep{ucar2015new,lee2016finite,murthy2015parallel,chaudhary2015web,zhang2016weak,jog2016time}, including local min orthogonal\citep{chen2015estimate} , bilinear finite element \citep{duran2013supercloseness} , multiscale finite element \citep{hou1997multiscale}, mixed multiscale finite element \citep{chen2003mixed} and generalized finite element \citep{melenk1996partition} methods.The special finite element method considered in this paper was introduced by Hollig in \cite{hollig2001weighted}- \citep{hollig2003finite}-\citep{hollig2005introduction}.It has been designed as an approximation of the solution by combining the advantages of B-splines and standard mesh-based trial functions.\\

\subsection{Organization of the paper}
The  article is organized as follows : In section (\ref{sec4}) we briefly provide some  description about WEBS-FEM. Section(\ref{sec3}) are devoted to the analysis of the boundary value problems  and regularity of weak solutions of the problem. We  demonstrates its convergence scheme for different boundary conditions in section (\ref{sec5}) and (\ref{sec6}) respectively.Finally the last two sections are devoted to the numerical experiments and conclusions.\\ Through out this paper following notational conventions have been used: the grid width used for the spline approximation is denoted by h, and for the functions f,g we write f $\preceq$ g, if f $\leq$ cg and f $\asymp$ g ,if f=cg  for some positive constant c which doesn't depend on the grid width , indices, or arguments of functions.

\section{Weighted extended b-splines Approximation Procedure}
\label{sec4}
In this study we focus on the development of WEB method for non cooperative elliptic systems of equations.First  we construct finite elements with scaled translates $b^n_k,k \in  \mathbb{Z}^m$ of the standard $m$-variate tensor product B-spline of order $n$ written as ,  $ b^n_k~ (b_k$ in short form ) which are polynomials of degree $n-1$ in the variables $x_1,x_2,x_3 , \ldots ,x_m$ on each grid cell $Q_l = h ([0,1]^m + l)~,~ k_i \leq l_i\leq k_i + n$ in their support~,~ cf; \cite{hollig2013approximation}-\citep{ciarlet2002finite}-\citep{de2013box} for more details and the references there in.

\begin{figure}[]
\begin{center} 
  \includegraphics[width=0.5\textwidth]{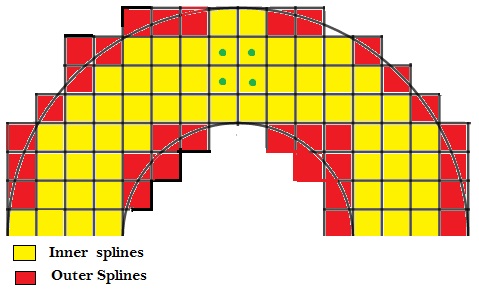}
\caption{Inner and Outer b splines.}
\end{center}
\end{figure}

\subsection{Splines on Bounded Domains}
The Splines $\mathbb{B}^n_h(D)$ on a bounded domain $D\subset \mathbb{R}^m$ consist of all linear combinations $\sum\limits_{k\in K} c_k b^n_{k,h}$ of relevant B-Splines; i.e, the set $K$ of relevant indices contains all $k$ with $b^n_{k,h}(x)\ne 0$ for some $x \in D$ ,where $b^n_{k,h}(x)= b^n (x/h-k)$ is the scaled translates and $h$ is the grid width. \\

The b-splines $b_k$ are piecewise polynomials on the $h$ -grid with vertices $h \mathbb{Z}^2$ and scaled so that the $L_2$ norm $||b_k||_0=||b||_0$ is independent of $h$.

\subsection{Inner and outer splines}
Grid cells $Q = h ([0,1]^m + l$) are partitioned into interior,  boundary and exterior cells depending on whether $Q\subseteq\bar{D}$ , the interior of $Q$ intersects $\partial D$ , or $Q \cap D = \phi $. Among the relevant B-Splines,$b_k,k \in  K$, a distinction is  made between inner B-splines
 $ b_i,i \in  I$ which have at least one interior cell in their support, and outer B-Splines
   $ b_j,j \in J=K \backslash I$
     for which supp $b_j$ consists entirely of boundary and exterior cells.
Now it is tempting to use $B_h := \text{span}\{b_k~:~k \in K\}$ as a finite element approximation space. But b splines may not conform to the boundary  conditions, but this situation can be overcome easily by multiplying with a suitable weight functions.

 \subsection{Weight functions} There are several different techniques for constructing weight functions Rvachev's R-function method provides weight functions \cite{rvachev1995r}.

\subsection{ Weighted extended b-spline :} 
\label{def1}
For $i \in  I$, the WEB-Spline $B_i$ is defined by,
\begin{center}
$B_i=\frac{w}{w(x_i)} \Big (b_i+\sum\limits_{j\in{J}}e_{i,j}b_j \Big)$ ,
\end{center}
where $x_i$ denotes the center of the grid cell $Q_{i+l(i)}$ corresponding  to $b_i$.The Lagrange coefficients $e_{i,j}$ satisfy,
\begin{center}
$|e_{i,j}| \preceq 1, e_{i,j}=0 ~\text{for}~ ||i-j||\succeq 1$
\end{center}

\begin{theorem}
If $\forall j \in J $
\begin{center}
$I(j)=\{ i\in  \mathbb{Z}^m : \alpha_\mu \leqq i_\mu\leqq \alpha_\mu+n-1 \}$
\end{center}
is a closest index array to $j$, then the coefficients,
$$  e_{i,j}=\left\{ \begin{array}{ll}
    \prod_{\mu=1}^m \prod_{\ell=\alpha_{\mu}}^{\alpha_{\mu}+n-1}\frac{j_{\mu}-\ell}{i_{\mu}-\ell} \hspace{2mm}for\hspace{1mm} i \in I(j),\\
    0 \hspace{10 mm} , else\\
      \end{array}
\right.$$
are admissible for constructing web-splines according to Definition (\ref{def1}).
\end{theorem}
 Details description of $e_{ij}$ ' s and its  requirement in construction of webs can be found in \cite{hollig2003finite}.

\section{Formulation of the problem}
\label{sec3}
\subsection{Finite element discretization}
Consider the following classical formulation of the reaction-diffusion problem in terms of a non cooperative elliptic system: Find  $\mathbf{u} = (u_1,u_2)$ such that,

\noindent\begin{minipage}{.5\linewidth}
\begin{align}
\begin{split}
  -\text{div}(\mathfrak{P}\nabla u_1)-\tau_2 u_2 = f_1 \\
 u_1=U_1\\
 \nu^t \cdot \mathfrak{P}\nabla u_1 =g_1\\
  \nu^t \cdot \mathfrak{P}\nabla u_1 + \sigma_1 u_1 =h_1
\end{split}
\label{eq20}
\end{align}
\end{minipage}%
\begin{minipage}{.5\linewidth}
\begin{align}
\begin{split}
  \text{div}(\mathfrak{P}\nabla u_2)+\tau_1 u_1 = f_2~~~,~~~in~\Omega, \\
  u_2=U_2~~~, ~~~on~\partial_1 \\
  -\nu^t \cdot \mathfrak{P}\nabla u_2 =g_2~~~, ~~~on \partial_2,\\
  -\nu^t \cdot \mathfrak{P}\nabla u_2 + \sigma_2 u_2 =h_2~~~, ~~~on~ \partial_3
\end{split}
\label{eq21} 
\end{align}
\end{minipage}
\vspace{5mm}
where $\Omega$ is a bounded domain in $\mathbb{R}^m$ with a sufficiently smooth boundary $\partial\Omega$ such that $\overline{\partial \Omega} = \overline{\partial_1} \cup \overline{\partial_2} \cup \overline{\partial_3}$,$\nu $ is the outward normal to the boundary, and $\partial_1,\partial_2~and~\partial_3$ are disjoint sets with positive measures.In order to formulate the above problem in the weak form, we further assume that the $u_i , f_i \in L_2(\Omega)$, $h_i  \in L_2(\partial_3)$ and $\sigma_i \in L_{\infty}(\partial_3)$. Then from regularity theorems  $u_i , U_i  \in H^2(\Omega)$.Moreover $\tau_j $'s  and $g_i$'s ($i$=1,2) are real essentially bounded as well as $\mathfrak{P}$  is a symmetric  $ m \times m$  matrix of real essentially bounded functions, and is uniformly positive definite,i.e,
\begin{equation}
C_1 |\xi|^2 \leq \mathfrak{P}(\mathbf{x})\xi \cdot \xi \leq C_2 |\xi|^2~~~ \forall \xi \in \mathbb{R}^m~~~~~a.e~ in~ \Omega
\nonumber
\end{equation}
where, $C_1~and~C_2$ are positive constants.In addition we assume $\tau_j  > 0$ a.e on $\Omega$\\
The weak formulation of problem (\ref{eq20})-(\ref{eq21}) then reads: Find $u_i \in U_i + H^1_{\partial_1}(\Omega)$ such that,

\begin{align}
\begin{split}
 (-1)^{i+1}\int_{\Omega}  \mathfrak{P}\nabla u_i\cdot \nabla v_i ~dx& + (-1)^i\int_{\Omega} \tau_j u_jv_i~dx + \int_{\partial_3} \sigma_i u_i v_i ~ds \\ 
 = &\int_{\Omega} f_iv_i~dx + \int_{\partial_2} g_i v_i ~ds + \int_{\partial_3} h_iv_i~ds ~~~\forall v_i \in H^1_{\partial_1}(\Omega)
 \end{split}
 \label{eq22}
\end{align}
where, $i,j = 1,2$ as well $i \ne j$ and,
\begin{align*}
H^1_{\partial_1}(\Omega) = \{ u \in H^1(\Omega) ~:~ u = 0 ~on~ \partial_1 \}
\end{align*}
Let us define a bilinear form $a(\cdot,\cdot)$ and a linear form $G(\cdot)$ as
\begin{align}
a(u_i,v_i)&= (-1)^{i+1} \int_{\Omega}  \mathfrak{P}(\nabla u_i)\cdot \nabla v_i ~dx  + (-1)^i\int_{\Omega} \tau_j u_jv_i~dx \nonumber \\&+ \int_{\partial_3} \sigma_i u_i v_i ~ds ~~,~u_i,v_i \in H^1(\Omega)
\end{align}
\begin{align}
G(v_i)=\int_{\Omega} f_iv_i~dx + \int_{\partial_2} g_i v_i ~ds + \int_{\partial_3} h_iv_i~ds ~~,~~v_i \in H^1(\Omega).
\end{align}
Then the weak formulation (\ref{eq22}) can be rewritten in a compact form: Find $u_i \in U_i + H^1_{\partial_1}(\Omega)$ such that $a(u_i,v_i)=G(v_i)~~ \forall~ v_i \in H^1_{\partial_1}(\Omega)$.
\subsection{Remark}
\label{rem1}
The above conditions on the cofficients of the problem provide the existence and uniqueness of the weak solution $u_i$ defined by (\ref{eq22}) due to the well known Lax-Milgram lemma \citep{brenner2007mathematical}. However it should be stressed here that this lemma is actually applied to the following problem: Find $u_{\star} \in H^1_{\partial_1}(\Omega)$ such that $a(u_{\star},v_i)= \overline{G}(v_i) ~~\forall~ v_i \in H^1_{\partial_1}(\Omega)$  where $\overline{G}(v_i)= G(v_i)-a(U_i,v_i)$,and setting  $ u_i := U_i+u_{\star}$ after that.\\

The energy functional $J$ of problem (\ref{eq22})is defined as,
\begin{equation}
J(v_i)= \frac{1}{2}a(v_i,v_i)- \overline{G}(v_i),~~~\forall~ v_i \in H^1(\Omega)
\end{equation}
It is well known that problem (\ref{eq22}) is equivalent to the problem of finding the minimizer ( which is equal to the above introduced solution $u_{\star}$ ) of the above energy functional $J$ over the set $H^1_{\partial_1}(\Omega)$ and setting again $u_i := U_i +u_{\star}$.\\
In what follows we need Friedrich's inequality,
\begin{equation}
||v_i||_{0,\Omega} \leq  C_{0,\Omega} || \nabla v_i || _{0,\Omega} ~~~~~\forall~ v_i \in H^1_{\partial_1}(\Omega)
\nonumber
\end{equation}
and the inequality from  the trace theorem,
\begin{equation}
||v_i||_{0,\partial\Omega} \leq  C ||v_i || _{1,\Omega} ~~~~~\forall~ v_i \in H^1(\Omega)
\nonumber
\end{equation}
where $  C_{0,\Omega} ~ and~ C$  are positive constants; see \citep{necas2013mathematical} for examples.

\section{Two-sided estimation of the error}
\label{sec5}
Let $u_i^h$ be a function from $U_i+\mathbb{B}_h \subset U_i+ H^1_{\partial_1}(\Omega)$, which we shall consider as an web-spline approximation of $u_i$. We can easily show that if $u^h_i := U_i + u^h_{\star}$,where $u^h_{\star}$ (depending on $i$) $\in \mathbb{B}_h \subset H^1_{\partial_1}(\Omega)$ then , 
\begin{align}
J(u^h_{\star})-J(u_{\star})= \frac{1}{2} a(u_i-u_i^h,u_i-u_i^h)
\label{e1}
\end{align} 
Thus,the main goal in what follows is to show how to effectively estimate the value ,
\begin{align}
a(u_i-u^h_i,u_i-u^h_i) = \int_{\Omega} \mathfrak{P} \nabla (u_i-u^h_i)\cdot \nabla (u_i-u^h_i)~dx &+ \int_{\Omega} \tau_j(u_j-u_j^h)  (u_i-u^h_i)~dx \nonumber \\+\int_{\partial_3}\sigma_i (u_i-u^h_i)^2~ds
\label{eq23}
\end{align}
from above and from below for an arbitrary approximation $u^h_i \in U_i + \mathbb{B}_h$.
\subsection{Upper bound for the error}
\label{subsec5}
We define,
\begin{align*}
||| \mathbf{u}|||_{\Omega} := (\int_{\Omega} \mathfrak{P}u \cdot u ~dx )^{1/2}~~~for ~ & u \in L_2(\Omega,\mathbb{R}^m)\\
H_{2,3}(div,\Omega) := \{ u^{\star} \in L_2(\Omega,\mathbb{R}^m) :~div~u^{\star} &\in L_2(\Omega) , \nu^t \cdot u^{\star} \in L_2(\partial_2 \cup \partial_3) \} \\
\tilde{\Omega}:= \{ x \in \Omega~ : \min_{j=1,2}{\tau_j (x)}\geq \tilde{\theta} > 0 \}
\end{align*}
\begin{theorem}
The following upper estimate holds for the  global error (\ref{eq23}).
\begin{align}
\begin{split}
a(\mathbf{u-u^h,u-u^h}) \preceq  \sum_{\substack{i,j=1 \\ i\neq j}}^2\biggl[\biggl|\biggl|\frac{1}{\sqrt{\tau_j }}\biggl(f_i + div~u_i^{\star} -\tau_j u_j^h \biggr) \biggr|\biggr|^2_{0,\tilde{\Omega}} + \biggl| \biggl| \frac{1}{\sqrt{\sigma_i}} \biggl( h_i - \sigma_i u^h_i - \nu^t \cdot u^{\star}_i \biggr)\biggr|\biggr|^2_{0,\partial_3}\\ +\bigl|\bigl|\bigl| \mathfrak{P}^{-1}u^{\star}_i - \nabla u_i^h\bigr|\bigr|\bigr|^2_{\Omega} + \bigl|\bigl| \chi_{\Omega \setminus \tilde{\Omega}} (f_i + div~ u^{\star}_i - \tau_j u_j^h) \bigr|\bigr|^2_{0,\Omega} \\+ \bigl|\bigl|g_i - \nu^t \cdot u^{\star}_i\bigr|\bigr|^2_{0,\partial_2} + \bigl|\bigl| \sqrt{\tau_j} (u_j-u_j^h)\bigr|\bigr|^2 \biggr] 
\end{split}
\label{eq24}
\end{align} 
where, $ u^{\star}_i \in H_{2,3}(div,\Omega)~and~ \mathbf{u}= (u_1,u_2)$.
\end{theorem}

\begin{proof}
Firstly, from (\ref{eq23})we notice ,
\begin{align*}
a(u_i-u^h_i,u_i-u^h_i)&= \le \bigl|\bigl|\bigl| \nabla (u_i-u_i^h) \bigr|\bigr|\bigr|^2_{\Omega} + \int_{\Omega} \tau_j (u_j-u_j^h)(u_i-u_i^h) \,dx+ \bigl|\bigl| \sqrt{\sigma_i} (u_i-u_i^h)\bigr|\bigr|^2_{0,\partial_3}\\
\\
&\le \bigl|\bigl|\bigl| \nabla (u_i-u_i^h) \bigr|\bigr|\bigr|^2_{\Omega} +\bigl|\bigl| \sqrt{\tau_j } (u_i-u_i^h)\bigr|\bigr|_{0,\tilde{\Omega}} \bigl|\bigl| \sqrt{\tau_j } (u_j-u_j^h)\bigr|\bigr|_{0,\tilde{\Omega}}\\&+ \bigl|\bigl| \sqrt{\sigma_i} (u_i-u_i^h)\bigr|\bigr|^2_{0,\partial_3}\\
\\
&\preceq \bigl|\bigl|\bigl| \nabla (u_i-u_i^h) \bigr|\bigr|\bigr|^2_{\Omega} +\bigl|\bigl| \sqrt{\tau_j } (u_i-u_i^h)\bigr|\bigr|_{0,\tilde{\Omega}}^2+ \bigl|\bigl| \sqrt{\tau_j } (u_j-u_j^h)\bigr|\bigr|_{0,\tilde{\Omega}}^2\\&+ \bigl|\bigl| \sqrt{\sigma_i} (u_i-u_i^h)\bigr|\bigr|^2_{0,\partial_3}\\
\end{align*}
Further from (\ref{eq22}) taking $v_i = u_i-u_i^h$ we obtain,
\begin{align}
\begin{split}
a(u_i-u^h_i,u_i-u^h_i)= \int_{\Omega} f_i(u_i - u^h_i)~dx + &\int_{\partial_2} g_i ( u_i - u_i ^h) ~ds + \int_{\partial_3} h_i (u_i -u_i^h)~ds  \\
- \int_{\Omega} \mathfrak{P} \nabla u_i^h & \cdot \nabla (u_i-u_i^h) ~dx + \int_{\Omega} \tau_j u_j^h (u_i-u_i^h)~dx \\&-\int_{\partial_3} \sigma_i u_i^h (u_i-u_i^h)~ds.
\end{split}
\label{eq25}
\end{align}
We have,
\begin{align}
\int_{\Omega} u_i^{\star} \cdot \nabla(u_i-u_i^h) ~dx + \int _{\Omega} div~ u_i^{\star} (u_i-u_i^h)~dx &= \int_{\partial_2 \cup \partial_3} \nu^t \cdot u_i^{\star} (u_i-u_i^h) ~ds~~,\\&~~\text{valid }\forall~ u^{\star}_i \in H_{2,3}(div,\Omega)\nonumber. 
\end{align}
Inserting this identity into equation (\ref{eq25}) we get ,
\begin{align*}
a(u_i-u^h_i,u_i-u^h_i)= &\int_{\Omega} \mathfrak{P} ( \mathfrak{P}^{-1} u^{\star}_i - \nabla u_i^h) \cdot \nabla (u_i-u_i^h)~dx - \int_{\Omega} (f_i + div~ u^{\star}_i - \tau_j u_j^h )(u_i-u_i^h)~dx \\ +& \int_{\partial_2} ( g_i - \nu^t \cdot u_i^{\star})(u_i-u_i^h) ~ds + \int_{\partial_3} ( h_i - \sigma_i u_i^h - \nu^t \cdot u_i^{\star} ) (u_i-u_i^h) ~ds\\
=&I_1 -I_2+I_3+I_4
\end{align*}
We now estimate each of these expressions. Clearly,
\begin{align*}
I_1 \leq \bigl|\bigl|\bigl| \mathfrak{P}^{-1} u^{\star}_i - \nabla u_h^i \bigr|\bigr|\bigr|_{\Omega}~ \bigl|\bigl|\bigl| \nabla ( u_i -u_i^h) \bigr|\bigr|\bigr|_{\Omega}.~~~\text{ from Cauchy-Schwarz's  Inequality}.
\end{align*}
Now decomposing $\Omega ~into~ \tilde{\Omega} ~and~ \Omega \setminus \tilde{\Omega}$ we estimate ,
\begin{align*}
I_2 \preceq  \frac{1}{2}\bigl|\bigl| \sqrt{\tau_j } (u_i-u_i^h) \bigr|\bigr|^2_{0,\tilde{\Omega}} +\frac{1}{2} \biggl|\biggl|\frac{1}{\sqrt{\tau_j }}(f_i + div~ u_i^{\star} - \tau_j u_j^h) \biggr|\biggr|^2_{0,\tilde{\Omega}}\\+\bigl|\bigl| \chi_{\Omega \setminus \tilde{\Omega}} (f_i + div~ u_i^{\star} - \tau_j u_j^h)\bigr|\bigr|_{0,\Omega}~ \bigl|\bigl|\bigl| \nabla (u_i-u_i^h)\bigr|\bigr|\bigr|_{\Omega}~~;~ab\leq \frac{1}{2}(a^2+b^2)
\end{align*}
Moreover,from  Friedrich's inequality and trace inequality   (\ref{rem1})  we get the following estimates,
\begin{align*}
I_3 \leq \big|\bigl| g_i - \nu^t \cdot u^{\star}_i \bigr|\bigr| _{0,\partial_2} ~ \bigl|\bigl| u_i-u_i^h \bigr|\bigr| _{0,\partial_2} \preceq \big|\bigl| g_i - \nu^t \cdot u^{\star}_i \bigr|\bigr| _{0,\partial_2}~ \bigl|\bigl|\bigl| \nabla (u_i-u_i^h)\bigr|\bigr|\bigr|_{\Omega}\\
I_4 \leq \frac{1}{2}\bigl|\bigl| \sqrt{\sigma_i}(u_i-u^h_i) \bigr|\bigr|^2_{0,\partial_3}+ \frac{1}{2}\biggl|\biggl| \frac{1}{\sqrt{\sigma_i}} \bigl(h_i-\sigma_i u_i^h - \nu^t \cdot u^{\star}_i\bigr) \biggr|\biggr|^2_{0,\partial_3}
\end{align*}
Combining all the above estimates  we conclude ,
\begin{align*}
a(u_i-u^h_i,u_i-u^h_i) &\preceq \biggl(\bigl|\bigl|\bigl| \mathfrak{P}^{-1} u^{\star}_i - \nabla u_h^i \bigr|\bigr|\bigr|_{\Omega} + \big|\bigl| g_i - \nu^t \cdot u^{\star}_i \bigr|\bigr| _{0,\partial_2} \\&+ \bigl|\bigl| \chi_{\Omega \setminus \tilde{\Omega}} (f_i + div~ u_i^{\star} - \tau_j u_j^h)\bigr|\bigr|_{0,\Omega} \biggr)\bigl|\bigl|\bigl| \nabla (u_i-u_i^h)\bigr|\bigr|\bigr|_{\Omega}\\&+\frac{1}{2}\bigl|\bigl| \sqrt{\sigma_i}(u_i-u^h_i) \bigr|\bigr|^2_{0,\partial_3}+ \frac{1}{2}\biggl|\biggl| \frac{1}{\sqrt{\sigma_i}} \bigl(h_i-\sigma_i u_i^h - \nu^t \cdot u^{\star}_i\bigr) \biggr|\biggr|^2_{0,\partial_3} \\&+ \frac{1}{2}\bigl|\bigl| \sqrt{\tau_j } (u_i-u_i^h) \bigr|\bigr|^2_{0,\tilde{\Omega}} +\frac{1}{2} \biggl|\biggl|\frac{1}{\sqrt{\tau_j }}(f_i + div~ u_i^{\star} - \tau_j u_j^h) \biggr|\biggr|^2_{0,\tilde{\Omega}}\\
&\preceq \frac{1}{2} \biggl(\bigl|\bigl|\bigl| \mathfrak{P}^{-1} u^{\star}_i - \nabla u_i^h \bigr|\bigr|\bigr|_{\Omega} + \big|\bigl| g_i - \nu^t \cdot u^{\star}_i \bigr|\bigr| _{0,\partial_2} \\&+ \bigl|\bigl| \chi_{\Omega \setminus \tilde{\Omega}} (f_i + div~ u_i^{\star} - \tau_j u_j^h)\bigr|\bigr|_{0,\Omega} \biggr)^2 + \frac{1}{2} \bigl|\bigl|\bigl| \nabla (u_i-u_i^h)\bigr|\bigr|\bigr|_{\Omega}^2\\& +\frac{1}{2}\bigl|\bigl| \sqrt{\sigma_i}(u_i-u^h_i) \bigr|\bigr|^2_{0,\partial_3}+ \frac{1}{2}\biggl|\biggl| \frac{1}{\sqrt{\sigma_i}} \bigl(h_i-\sigma_i u_i^h - \nu^t \cdot u^{\star}_i\bigr) \biggr|\biggr|^2_{0,\partial_3} \\&+ \frac{1}{2}\bigl|\bigl| \sqrt{\tau_j } (u_i-u_i^h) \bigr|\bigr|^2_{0,\tilde{\Omega}} +\frac{1}{2} \biggl|\biggl|\frac{1}{\sqrt{\tau_j }}(f_i + div~ u_i^{\star} - \tau_j u_j^h) \biggr|\biggr|^2_{0,\tilde{\Omega}}
\end{align*}
Simplifying, we immediately obtain ,
\begin{align*}
& \bigl|\bigl|\bigl| \nabla (u_i-u_i^h) \bigr|\bigr|\bigr|^2_{\Omega} + \bigl|\bigl| \sqrt{\tau_j } (u_i-u_i^h)\bigr|\bigr|^2_{0,\tilde{\Omega}} + \bigl|\bigl| \sqrt{\sigma_i} (u_i-u_i^h)\bigr|\bigr|^2_{0,\partial_3}\\
&\preceq \biggl|\biggl|\frac{1}{\sqrt{\tau_j }}(f_i + div~ u_i^{\star} - \tau_j u_j^h) \biggr|\biggr|^2_{0,\tilde{\Omega}} + ~ \biggl|\biggl| \frac{1}{\sqrt{\sigma_i}} \bigl(h_i-\sigma_i u_i^h - \nu^t \cdot u^{\star}_i\bigr) \biggr|\biggr|^2_{0,\partial_3}\\&+ \biggl(\bigl|\bigl|\bigl| \mathfrak{P}^{-1} u^{\star}_i - \nabla u_h^i \bigr|\bigr|\bigr|_{\Omega} + \big|\bigl| g_i - \nu^t \cdot u^{\star}_i \bigr|\bigr| _{0,\partial_2} \\&+ \bigl|\bigl| \chi_{\Omega \setminus \tilde{\Omega}} (f_i + div~ u_i^{\star} - \tau_j u_j^h)\bigr|\bigr|_{0,\Omega} \biggr)^2
\end{align*}
\begin{align*}
\i.e, a(u_i-u_i^h,u_i-u_i^h)& \preceq \biggl|\biggl|\frac{1}{\sqrt{\tau_j }}(f_i + div~ u_i^{\star} - \tau_j u_j^h) \biggr|\biggr|^2_{0,\tilde{\Omega}} + ~ \biggl|\biggl| \frac{1}{\sqrt{\sigma_i}} \bigl(h_i-\sigma_i u_i^h - \nu^t \cdot u^{\star}_i\bigr) \biggr|\biggr|^2_{0,\partial_3}\\&+ \biggl(\bigl|\bigl|\bigl| \mathfrak{P}^{-1} u^{\star}_i - \nabla u_h^i \bigr|\bigr|\bigr|_{\Omega} + \big|\bigl| g_i - \nu^t \cdot u^{\star}_i \bigr|\bigr| _{0,\partial_2} \\&+ \bigl|\bigl| \chi_{\Omega \setminus \tilde{\Omega}} (f_i + div~ u_i^{\star} - \tau_j u_j^h)\bigr|\bigr|_{0,\Omega} \biggr)^2+\bigl|\bigl| \sqrt{\tau_j} (u_j-u_j^h) \bigr|\bigr|^2
\end{align*}
Now using the obvious inequality ,
\begin{align*}
(a+b)^2 \leq \ 2 (a^2+b^2)
\end{align*}
and from  the definition ,
\begin{align*}
a(\mathbf{u}-\mathbf{u}_h , \mathbf{u} - \mathbf{u}_h) = \sum_{i=1}^2a(u_i-u^h_i , u_i - u^h_i) 
\end{align*}
we finally  get  the  following estimate (\ref{eq24}).
\end{proof}
\subsection{Lower estimate for the error}
\begin{theorem}
For the error in the energy norm (\ref{eq23}) we have the lower bound ,
\begin{align}
a(\mathbf{u-u}^h,\mathbf{u-u}^h) \geq 2 \bigl(J(\mathbf{u}^h_{\star}) -J(\mathbf{v}_{\star})\bigr) ~~~~~\forall~ \mathbf{v}_{\star} \in (H^1_{\partial_1})^2
\label{eq26}
\end{align}
\begin{proof}
We have from (\ref{e1}) ,
\begin{align*}
2\bigl( J(\mathbf{u^h_{\star}}) - J(\mathbf{u_{\star}}) \bigr)&=a(\mathbf{u^h_{\star}},\mathbf{u^h_{\star}}) -2\overline{G}(\mathbf{u^h_{\star}})-a(\mathbf{u_{\star}},\mathbf{u_{\star}}) + 2 \overline{G}(\mathbf{u_{\star}})\\
&= a(\mathbf{u^h_{\star}},\mathbf{u^h_{\star}}) - a(\mathbf{u_{\star}},\mathbf{u_{\star}}) +2a(\mathbf{u_{\star}},\mathbf{u_{\star}})-2a(\mathbf{u_{\star}}, \mathbf{u^h_{\star}})\\
&= a(\mathbf{u^h_{\star}},\mathbf{u^h_{\star}}) - a(\mathbf{u_{\star}},\mathbf{u_{\star}}) + 2a(\mathbf{u_{\star}},\mathbf{u_{\star}}- \mathbf{u^h_{\star}})\\
&=a(\mathbf{u^h_{\star}},\mathbf{u^h_{\star}}) + a(\mathbf{u_{\star}},\mathbf{u_{\star}})-2a(\mathbf{u_{\star}},\mathbf{u^h_{\star}})\\
&= a(\mathbf{u_{\star}}- \mathbf{u^h_{\star}},\mathbf{u_{\star}}- \mathbf{u^h_{\star}})\\
&= a(\mathbf{u-u}_h,\mathbf{u-u}_h)\\
\end{align*}
In as much as $\mathbf{u_{\star}}$ minimizes the energy functional, we have $J(\mathbf{u_{\star}})\leq J(\mathbf{v}_{\star})$ for any $\mathbf{v}_{\star}$ in $(H^1_{\partial_1})^2$ and consequently it yields (\ref{eq26}).
\end{proof}
\end{theorem}

Preceding estimates can be extended in more general context.
\section{Type of equations and corresponding different boundary conditions}
\label{sec6}
\subsection{Dirichlet boundary condition :}In this case $\partial_3 ~and~ \partial_2 = \emptyset$, hence the second and last terms on the right hand side of estimate  (\ref{eq24}) do not exist, so we obtain the following version of estimate:
\begin{align*}
a(\mathbf{u-u^h,u-u^h}) \preceq \sum_{i=1}^2 \biggl[\biggl|\biggl|\frac{1}{\sqrt{\tau_j }}\biggl(f_i + div~u_i^{\star} -\tau_j  u_h \biggr) \biggr|\biggr|^2_{0,\tilde{\Omega}} +\bigl|\bigl|\bigl| \mathfrak{P}^{-1}u^{\star}_i - \nabla u_i^h\bigr|\bigr|\bigr|^2_{\Omega} \\+ \bigl|\bigl| \chi_{\Omega \setminus \tilde{\Omega}} (f_i + div~ u^{\star}_i - \tau_j u_j^h) \bigr|\bigr|^2_{0,\Omega}\biggr]
\end{align*}

\subsection{Dirichlet/Neumann boundary condition:}In this case $\partial_3=\emptyset$ and we obtain following variant of estimate:
\begin{align*}
a(\mathbf{u-u^h,u-u^h}) \preceq \sum_{i=1}^2 \biggl[\biggl|\biggl|\frac{1}{\sqrt{\tau_j }}\biggl(f_i + div~u_i^{\star} -\tau_j  u_h \biggr) \biggr|\biggr|^2_{0,\tilde{\Omega}} +\bigl|\bigl|\bigl| \mathfrak{P}^{-1}u^{\star}_i - \nabla u_i^h\bigr|\bigr|\bigr|^2_{\Omega} &\\+ \bigl|\bigl| \chi_{\Omega \setminus \tilde{\Omega}} (f_i + div~ u^{\star}_i - \tau_j u_j^h) \bigr|\bigr|^2_{0,\Omega} + \bigl|\bigl|g_i - \nu^t \cdot u^{\star}_i\bigr|\bigr|^2_{0,\partial_2} \biggr] 
\end{align*} 
\subsubsection{Remark}
For such a mixed boundary condition one needs, in general, to compute global constants depending on $\Omega,\partial_1~and~\partial \Omega$, in case of $\tau_j =0$ preceding estimate presented in  \citep{repin2004posteriori}. But our  way of obtaining is much simpler and more straightforward.

\section{Numerical example}
\label{sec7}
In this section we present a  numerical example to verify the analysis illustrated  in the previous sections. 

\subsection{Implementation} Let the quadrangulation $\mathcal{T}_h$ be the collection of all cells $K \subset \Omega$ such that $K \cap \partial \Omega = \emptyset$. 
We define the finite element space $X_h$ as follows :
\begin{align*}
X_h : = \left \{ (u_1,u_2) ~:~u_k=\sum_{j \in I^u} \alpha_j^k \psi_j~,~\alpha_j^k \in \mathbb{R}~, k= 1,2 \right\}
\end{align*}
where $I^u$ is the inner nodes. For $ i \in I^u$ , let $\Psi_i$ be the web spline of order $3$ which is the tensor product of scled translate of the function  $\psi$ defined :
\begin{align*}
\psi(x) =
\left\{
	\begin{array}{ll}
		\frac{x^2}{2}  & \mbox{on } [0,1), \\
		\frac{1}{2} +(x-1) - (x-1)^2 & \mbox{on } [1,2),\\
		\frac{(x-3)^2}{2}  & \mbox{ on} [2,3)
	\end{array}
\right.
\end{align*}
and
\begin{align*}
X_{\tilde{\psi}} := \oplus _{J \in \mathcal{T}_h} X_J
\end{align*}
where, $X_J$ is  the one dimensional subspace spanned by the function $\tilde{\psi}_J$  defined
\begin{align*}
\tilde{\psi}_J (x,y) := \omega (x,y)\psi \left(\frac{x}{h} -\ell_0 \right) \psi \left(\frac{y}{h} -\ell_1 \right) 
\end{align*}
We choose the weight function  $\omega(x,y)= (x^2+y^2 - 1)(4-x^2-y^2)$. Therefore the approximation space is taken to 
\begin{align*}
X^2_h : = X_h \oplus X_{\tilde{\psi}}
\end{align*} 

Again the solution lies in the set of admissible functions $H^1_{g} := \{ u \in H^1(\Omega) : u=g ~\text{ on } \partial \Omega\}$ which is clearly not vanishes at the boundary. In order to describe this set we choose $u_g $ ( i.e, $u$ takes the value $g$ on the boundary)  and consequently the solution is determined by $u_g+u_0$ with $u_0 \in H_0^1$. Finally,
\begin{align*}
H^1_g =u_g \oplus H^1_0
\end{align*}

\subsection{Residual for the partial differential equation} The residual of the Ritz-Galerkin approximation $\mathbf{u}_h$ is defined as ,
\begin{equation*}
\mathbf{r}(x,y)= (L\mathbf{u}_h)(x,y) - \mathbf{f}(x,y)
\end{equation*}

where , L is the differential operator of the boundary value problem. The relative error,

\begin{equation*}
\epsilon_{\text{Res}} = ||\mathbf{r}||_{0,\Omega}/||\mathbf{f}||_{0,\Omega}
\end{equation*}

with $||.||_{0,\Omega}$  denoting the $L_2$ norm on $\Omega$ provides a measure of accuracy for the solution
without having to resort to grid refinement.\\

In  the  following example we have a quadrant domain  $$\Delta = \{(x,y) : 0\le x^2+y^2  \le 1 , x,y \ge 0\}$$   and $ n = 2$ . Furthermore ,we use an SSOR-preconditioned cg-iteration  with  the following convergence  stopping criteria. 
\begin{align*}
\epsilon_{\text{Res}} =\mathcal{O}(10^{-6} )
\end{align*}
In Table (\ref{tab1}) we calculate the relative error $\epsilon$ and  Figure (\ref{fig1}) shows the numerical solution as well the residual equation at $h = 2^{-4}$.

\subsection{Population-dynamics systems} 

Consider the following equations that models the steady state solution of the population which is subdivided into two populations adults and children. The variables $u_1$ and $u_2$ represents the concentration of the populations of the adults and the children respectively. The parameters $\mathfrak{P}$ gives the rate of which a child becomes adult as well  the birth rate of the children ,$\tau_k$ represents the decrease in population size because of overcrowding effects and the death rate of the adults. Interested  readers are encouraged to go through \cite{mittelmann2000solving},\cite{doubova2016extinction} ,\cite{lin2004coexistence} and the references there in for details descriptions. \\

The model coefficients $\mathfrak{P}=
 \begin{bmatrix}
 x^2y & y\\
 y & y
 \end{bmatrix}
$ 

\begin{align*}
- \textrm{div} (\mathfrak{P}  \nabla u_1 )-0.05y ~u_2 &= -e^{x+y}\\
 \textrm{div} (\mathfrak{P}  \nabla u_2 )+2 x^2 ~u_1 &= -e^{x+y}
\end{align*}
 
 $\bullet$ \enquote{Mixed boundary conditions of Dirichlet-Neumann type} :\\

\begin{align*}
u_1&=y=u_2 ~~\mbox{ on }~~ x=0\\
 \nu^t \cdot \mathfrak{P}\nabla u_1 &=1=\nu^t \cdot \mathfrak{P}\nabla u_2~~ \mbox{ on }~~ y =0 \\
  \nu^t \cdot \mathfrak{P}\nabla u_1 + \sigma_1 u_1 &=0= \nu^t \cdot \mathfrak{P}\nabla u_2 + \sigma_2 u_2 ~~\mbox{ on } ~~x^2+y^2=1
\end{align*}

 \vspace{5mm}
 
 \begin{figure}[]
\begin{center} 
  \includegraphics[width=0.95\textwidth]{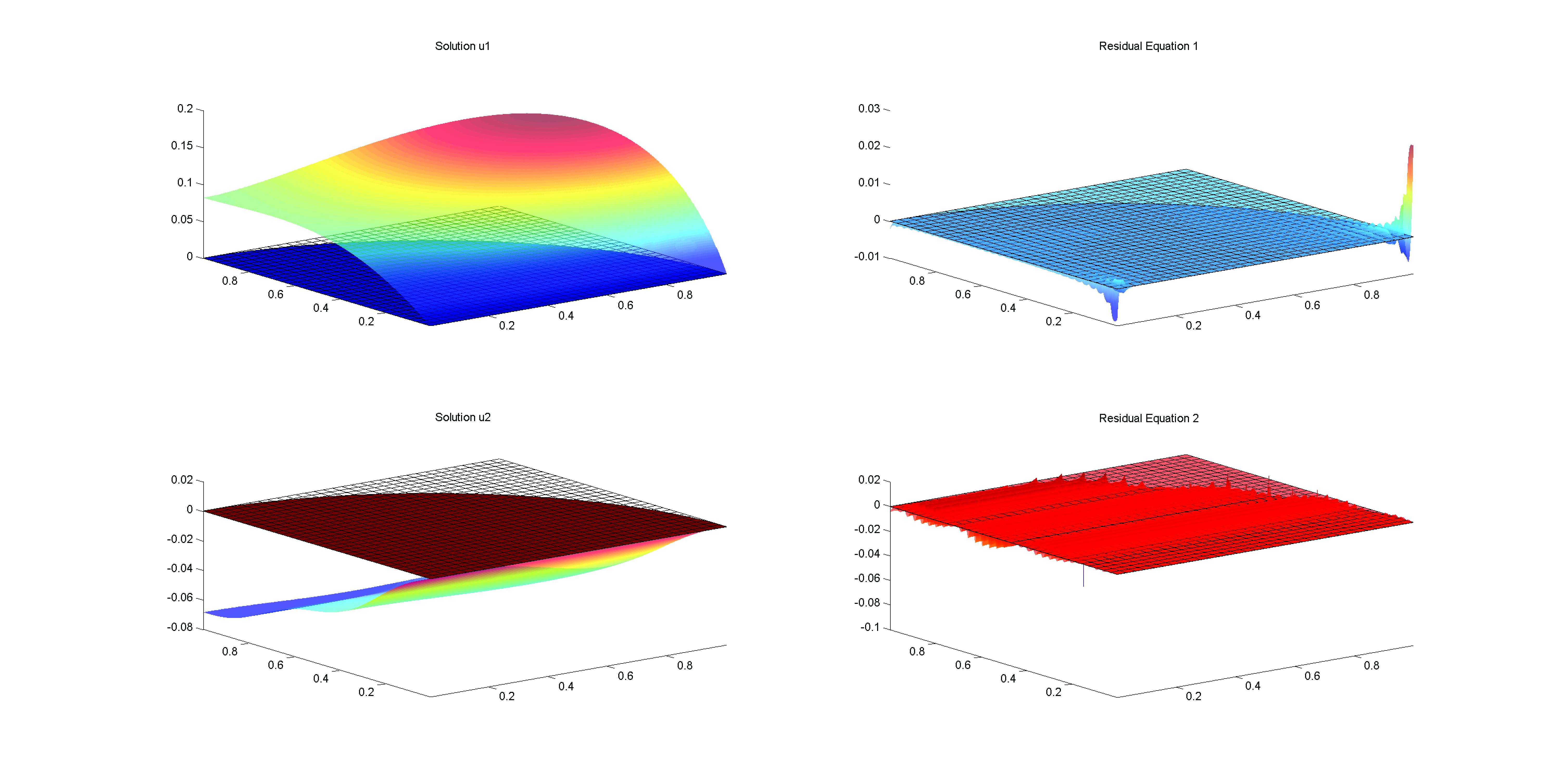}
\caption{Numerical Solution and residual equation.}
\end{center}
\label{fig1}
\end{figure}

\begin{center}
\begin{tabular}{ |p{10cm}||p{3cm}|  }
 \hline
 \multicolumn{2}{|c|}{Table  : Residual Errors} \\
 \hline
 $h$ & $\epsilon$ \\
 \hline
 $2^{-1}$   &$8.417 \times 10^{-2}$   \\
 
 $2^{-2}$   &$9.004 \times 10^{-4}$ \\
 
 $2^{-3}$  &$2.018 \times 10^{-6}$ \\
 
$2^{-4}$    &$1.005 \times 10^{-6}$ \\

 \hline
\end{tabular} 
\label{tab1}
\end{center}
\vspace{5mm}

 \section{Conclusion and perspectives} The webs method has a very wide range of applications. The most interesting one from the computational point of view, is that it is completely constrained. As a future extension, we propose to investigate the numerical analysis of the webs based isogeometric method , which involves the analysis of time-dependent convection-diffusion equations. Numerical simulations with this method give very promising results. Another extensions concerns subdivisions with nonmatching grids. Concerning our future work, our intention is to focus more on reaction-diffusion systems and non linear problems in the frame work of weighted isogeometric analysis. We would like to improve the used method based on b-splines by extending into rational b-splines. By using this method it should be possible to encapsulate the exact geometry representation in the analysis  while providing geometric  flexibility.

\end{document}